\documentclass{article}
\usepackage{amsmath}
\usepackage{amssymb}
\usepackage{amsthm}
\usepackage{fullpage}
\usepackage[utf8]{inputenc}

\newcommand{\bv}{\mathbf{v}}

\newtheorem{theorem}{Theorem}[section]
\newtheorem{lemma}[theorem]{Lemma}
\newtheorem{definition}[theorem]{Definition}
\newtheorem{prop}[theorem]{Proposition}
\newtheorem{conjecture}{Conjecture}
\newtheorem{cor}[theorem]{Corollary}

\allowdisplaybreaks

\title{A Generalization of the Graham-Pollak Tree Theorem to Steiner Distance}
\author{Joshua Cooper and Gabrielle Tauscheck}
\date{\today}

\begin{document}

\maketitle

\begin{abstract}
Graham and Pollak (\cite{GraPol71}) showed that the determinant of the distance matrix of a tree $T$ depends only on the number of vertices of $T$.  Graphical distance, a function of pairs of vertices, can be generalized to ``Steiner distance'' of sets $S$ of vertices of arbitrary size, by defining it to be the fewest edges in any connected subgraph containing all of $S$.  Here, we show that the same is true for trees' {\em Steiner distance hypermatrix} of all odd orders, whereas the theorem of Graham-Pollak concerns order $2$.  We conjecture that the statement holds for all even orders as well.
\end{abstract}

\section{Introduction}


Graham and Pollak showed that the determinant of the distance matrix of a tree $T$ on $n$ vertices -- the $n \times n$ matrix whose each $(v,w) \in V(T)\times V(T)$ entry is the ordinary graph distance between $v$ and $w$ -- depends only on $n$.  In fact, they gave a formula: $-(n-1)(-2)^{n-2}$.  Y.~Mao asks\footnote{Personal communication.} whether this result can be extended to ``Steiner distance'', a generalization of distance introduced by Hakimi \cite{Hak71} and popularized by \cite{ChaOelTiaZou89}. The Steiner distance $d_G(S)$ of a set $S \subseteq V(G)$ of vertices is the fewest number of edges in a connected subgraph of $G$ containing all of $S$.  Note that, if $S = \{v,w\}$, this reduces to the classical definition of the distance from $v$ to $w$, since a connected graph of smallest size containing $v$ and $w$ is a path of length $d_G(v,w) := d_G(\{v,w\})$. (See \cite{Mao17} for an extensive survey on Steiner distance.)  Here, we show that the result of Graham-Pollak extends to {\em Steiner distance hypermatrices}, at least for odd orders.  Furthermore, we describe the structure of the set of nullvectors for order $3$, a projective variety of codimension $2$, showing along the way that the sum of the coordinates of any nullvector is zero.

Just as all pairwise distances in a graph can be represented by a symmetric matrix, we can write the Steiner distances of all $k$-tuples of vertices as an order-$k$ {\em hypermatrix} (sometimes referred to as a {\em tensor}): the $\overbrace{[n] \times \cdots \times [n]}^k$ (super-)symmetric integer array whose $(v_1,\ldots,v_k)$ entry is the Steiner distance of $\{v_1,\ldots,v_k\}$.  We sometimes refer to such hypermatrices as ``cubical'' since all the index sets are identical. There is a notion of hyperdeterminant that generalizes determinant, and shares many of its properties, though in general is much harder to compute.  See, for example, \cite{Qi05} for discussion of the symmetric hyperdeterminant.  For our purposes, what will matter about the hyperdeterminant is that it detects nontrivial simultaneous vanishing of a system of degree-$(k-1)$ homogeneous polynomials (aka $(k-1)$-forms) in $n$ variables, as the following result makes precise:

\begin{theorem}[\cite{GelKapZel08} Theorem 1.3] \label{thm:hyperdeterminant} The hyperdeterminant $\det(M)$ of the order-$k$, dimension-$n$ hypermatrix $M = (M_{i_1,\ldots,i_k})_{i_1,\ldots,i_k=1}^n$ is a monic irreducible polynomial which evaluates to zero iff there is a nonzero simultaneous solution to $\nabla f_M = \vec{0}$, where
$$
f_M(x_1,\ldots,x_n) = \sum_{i_1,\ldots,i_k} M_{i_1,\ldots,i_k} \prod_{j=1}^k x_{i_j}.
$$
\end{theorem}


Note that there is a choice to be made in generalizing distance matrices: instead of $d_G(S)$, we could also simply set the entries corresponding to vertex sets $S$ of cardinality less than $k$ to zero.  However, doing so yields a hyperdeterminant of zero {\em irrespective of the non-degenerate entries}, as we now show. For a hypermatrix $M \in \mathbb{C}^{S \times \cdots \times S}$, call an entry $M(i_1,\ldots,i_k)$ ``degenerate'' if $|\{i_1,\ldots,i_k\}| < k$.

\begin{theorem}
Let $M$ be any cubical hypermatrix with all degenerate entries set equal to $0$. Then the hyperdeterminant of $M$ is $0$. 
\end{theorem}

\begin{proof}
To prove the hyperdeterminant is $0$, we exhibit a nontrivial simultaneous zero of the partial derivatives of the $k$-form 
$$
f_{M}(x) = \sum_{i_1,i_2,\ldots,i_k = 1}^n a_{i_1i_2\cdots i_k}x_{i_1}x_{i_2}\cdots x_{i_k}.
$$ 
Since $M$ has degenerate entries set equal to zero, any term that has $i_p = i_q$ for some $p,q \in [k]$ will have a matrix entry of zero and thus will not appear in the polynomial. Therefore, the only terms that will appear are $x_{i_1}x_{i_2}\cdots x_{i_k}$ with each $i_p$ distinct. The gradient vector of these polynomials will consist of terms of degree $k-1$ where once again each $i_p$ is distinct. Therefore, choose $x_{i_1} = x_{i_2} = \cdots = x_{i_{k-1}} = 0$ and let $x_{i_k}$ be any nonzero value; this is a nontrivial point where all partial derivatives vanish, so that the hyperdeterminant is $0$.  
\end{proof}

So, instead, we use Steiner distance to populate all entries of the hypermatrix.  This is made precise as follows.

\begin{definition} Given a graph $G$ and a subset $S$ of the vertices, the {\em Steiner distance} of $S$, written $d_G(S)$ or $d_G(v_1,\ldots,v_k)$ where $S = \{v_1,\ldots,v_k\}$, is the number of edges in the smallest connected subgraph of $G$ containing $S=\{v_1,\ldots,v_k\}$. Since such a connected subgraph of $G$ witnessing $d_G(S)$ is necessarily a tree, it is called a {\em Steiner tree} of $S$. 
\end{definition}

\begin{definition} Given a graph $G$, the {\em Steiner polynomial} of $G$ is the $k$-form 
$$
p^{(k)}_G(\mathbf{x}) = \sum_{v_1,\ldots,v_k \in V(G)} d_G(v_1,\ldots,v_k) x_1 \cdots x_k
$$
where we often suppress the subscript and/or superscript on $p^{(k)}_G$ if it is clear from context.
\end{definition}

Equivalently, we could define the Steiner $k$-form to be the $k$-form associated with the Steiner hypermatrix:

\begin{definition} Given a graph $G$, the {\em Steiner $k$-matrix} (or just ``Steiner hypermatrix'' if $k$ is understood) of $G$ is the order-$k$, cubical hypermatrix $\mathcal{S}_G$ of dimension $n$ whose $(v_1,\ldots,v_k)$ entry is $d_G(v_1,\ldots,v_k)$.
\end{definition}

Throughout the sequel, we write $D_r$ for the operator $\partial / \partial x_r$, and we always assume that $T$ is a tree.

\begin{definition} Given a graph $G$ on $n$ vertices, the {\em Steiner $k$-ideal} -- or just ``Steiner ideal'' if $k$ is clear -- of $G$ is the ideal in $\mathbb{C}[x_1,\ldots,x_n]$ generated by the polynomials $\{D_j p_G \}_{j=1}^n$.
\end{definition}

Thus, the Steiner ideal is the {\em Jacobian ideal} of the Steiner polynomial of $G$.  

\begin{definition} A {\em Steiner nullvector} is a point where all the polynomials within the Steiner ideal vanish. The set of all Steiner nullvectors -- a projective variety -- is the {\em Steiner nullvariety}.
\end{definition}

Although all the results contained herein concern odd order Steiner hypermatrices, extensive computation suggests that they extend to even order.

\begin{conjecture}
    The order-$k$ Steiner distance hypermatrix of a tree $T$ on $n\geq 3$ vertices has a hyperdeterminant that only depends on $T$ through $n$, and is $0$ iff $k$ is odd.
\end{conjecture}

Below, we show that this conjecture holds for all odd $k$, when the hyperdeterminant is $0$ irrespective of the choice of $T$.  We then go on to describe the Steiner nullvariety for $k=3$.

\section{Main Results}

\begin{theorem}
    For $k$ odd, the Steiner distance $k$-matrix of a tree $T$ with at least $3$ vertices has a hyperdeterminant equal to zero. 
\end{theorem}
\begin{proof}
    Since $T$ has at least $3$ vertices, let $u$ be a leaf, $w$ a neighbor of $u$, and $v \neq u$ a neighbor of $w$.  Let $\mathbf{y}$ denote the vector whose $z$ coordinate $y_z$ is given by
    $$
    y_z = \left \{ \begin{array}{ll} 
    1 & \textrm{ if } z=u \\
    \zeta & \textrm{ if } z=v \\
    -1-\zeta & \textrm{ if } z=w \\
    0 & \textrm{ otherwise,} 
    \end{array} \right .
    $$
    where $\zeta = \exp(\pi i/(k-1))$, a $(2k-2)$-root of unity. By Theorem \ref{thm:hyperdeterminant}, it suffices to show that $D_z p_T (\mathbf{y}) = 0$ for each $z \in V(T)$.  First, suppose $v$ is not on the $u-z$ path in $T$ and $z \neq u$ (which includes the case $z=w$).  Let $\alpha = d_T(z,u,v,w)$, so that
    \begin{align*}
        \frac{1}{k} D_z p_T (\mathbf{y}) &= \sum_{a+b+c=k-1} x_u^a x_v^b x_w^c \binom{k-1}{a,b,c} d_T(z,u,v,w) \\
        & + \sum_{a+c=k-1} x_u^a x_w^c \binom{k-1}{a,c} (d_T(z,u,w)-d_T(z,u,v,w)) \\
        & + \sum_{b+c=k-1} x_v^b x_w^c \binom{k-1}{b,c} (d_T(z,v,w)-d_T(z,u,v,w)) \\
        & + x_w^{k-1} (d_T(z,u,v,w)-d_T(z,u,w)-d_T(z,v,w)+d_T(z,w)) \\
        &= \alpha (x_u+x_v+x_w)^{k-1} - (x_u+x_w)^{k-1} - (x_v+x_w)^{k-1} \\
        &= 0 - (-\zeta)^{k-1} - (-1)^{k-1} = 0.
    \end{align*}
    Next, if $v$ is on the $u-z$ path in $T$ and $z \not \in \{u,w\}$, we obtain
    \begin{align*}
        \frac{1}{k} D_z p_T (\mathbf{y}) &= \sum_{a+b+c=k-1} x_u^a x_v^b x_w^c \binom{k-1}{a,b,c} d_T(z,u,v,w) \\
        & + \sum_{b+c=k-1} x_v^b x_w^c \binom{k-1}{b,c} (d_T(z,v,w)-d_T(z,u,v,w)) \\
        & + x_v^{k-1} (d_T(z,v)-d_T(z,v,w))\\
        &= \alpha (x_u+x_v+x_w)^{k-1} - (x_v+x_w)^{k-1} - x_v^{k-1} \\
        &= 0 - (-1)^{k-1} - (-\zeta)^{k-1} = 0.
    \end{align*}
    Finally, if $z=u$, then
    \begin{align*}
        \frac{1}{k} D_z p_T (\mathbf{y}) &= \sum_{a+b+c=k-1} x_u^a x_v^b x_w^c \binom{k-1}{a,b,c} d_T(u,v,w) \\
        & + \sum_{a+c=k-1} x_u^a x_w^c \binom{k-1}{a,c} (d_T(u,w)-d_T(u,v,w)) \\
        & + x_u^{k-1} (d_T(u)-d_T(u,w))\\
        &= 2 (x_u+x_v+x_w)^{k-1} - (x_u+x_w)^{k-1} - x_u^{k-1} \\
        &= 0 - (-\zeta)^{k-1} - 1^{k-1} = 0.
    \end{align*}
\end{proof}

Note that the hyperdeterminant of a tree on one vertex is also zero. This is because the Steiner $k$-form, $p_T^{(k)}$, only contains one monomial: $d_T(1,\ldots,1)x_1^k$. Since $d_T(1,\ldots,1) = 0$, the Steiner $k$-form as well as the partial derivative is automatically $0$, and so the Steiner nullvector $\mathbf{v} = (x_1)$ can be set to anything. 

For the tree on two vertices, the hyperdeterminant is not zero. It is straightforward to write the Steiner $k$-form as $p_T^{(k)} = (x_1+x_2)^k - x_1^k - x_2^k$. The partial derivatives are therefore $D_j p_T^{(k)} = k(x_1+x_2)^{k-1}-kx_j^{k-1}$ for $j=1,2$, so $D_1p_T^{(k)}=D_2p_T^{(k)}=0$ implies $x_1^{k-1}=(x_1+x_2)^{k-1}=x_2^{k-1}$.  Then $x_2 = \zeta x_1$, where $\zeta^{k-1} = 1$, but if $x_1 \neq 0$ this implies
$$
1 = (x_1+x_2)^{k-1}/x_1^{k-1} = (1 + \zeta)^{k-1},
$$
a contradiction.  Therefore, there is no nontrivial nullvector and the Steiner hyperdeterminant of the tree on two vertices is nonzero.  

Now that we have established that all Steiner hyperdeterminants of odd order with $n \geq 3$ are zero, we describe in more detail the corresponding Steiner nullvariety, at least for order $k=3$.

\begin{lemma} \label{lem:pathorY} For any distinct vertices $i,j,k$ of a tree $T$, we have
$$
2d_T(i,j,k) = d_T(i,j)+d_T(i,k)+d_T(j,k).
$$
\end{lemma}
\begin{proof}
    It is easy to check the formula for each of the two cases: either the Steiner tree of $\{i,j,k\}$ is a path or a tree with three leaves.
\end{proof}

The following result shows that $p^{(3)}_T$ is divisible by the elementary symmetric polynomial of degree $1$, which we refer to as $s$.

\begin{prop}\label{divisible}
    The Steiner 3-form $p^{(3)}_T$ is divisible by $s = \sum_r x_r$. 
\end{prop}

\begin{proof}
    Let $p = p^{(3)}_T$. If $s$ divides $p$, then $p = sg$ for some polynomial $g$. We claim $g = 3\sum_{i<j}d_T(i,j)x_ix_j$. We show that 
    $$
    p = sg = \sum_r x_r \left(3\sum_{i<j}d_T(i,j)x_ix_j\right) = \sum_{r,i<j}  3d_T(i,j)x_ix_j x_r.
    $$
    holds by classifying summands according to the triple $(r,i,j)$. 

    \begin{itemize}
        \item If $r = i$, the contribution is $3\sum_{i<j}d_T(i,j)x_i^2x_j$. 
        \item If $r = j$, the contribution is $3\sum_{i<j}d_T(i,j)x_ix_j^2$.
        \item If $r \neq i,j$, then the contribution becomes 
        \begin{align*}
        3\sum_{\substack{i<j \\ r \neq i,j}} d_T(i,j)x_ix_jx_r &= 3 \sum_{i<j<k} [d_T(i,j) + d_T(i,k) + d_T(j,k)] x_ix_jx_k \\
        &= 6 \sum_{i<j<k} d_T(i,j,k) x_ix_jx_k \\
        &= \sum_{i,j,k \text{ distinct}} d_T(i,j,k) x_ix_jx_k 
        \end{align*}
    \end{itemize}
    where the second equality follows from Lemma \ref{lem:pathorY}. On the other hand,
    $$
    p = 3 \sum_{i \neq j} d_T(i,j) x_i^2 x_j + \sum_{i,j,k \text{ distinct}} d_T(i,j,k) x_i x_j x_k,
    $$
    which agrees with the sum of the three types of terms in $sg$. 
\end{proof}

\begin{theorem}
    If $\bv = (x_1,\cdots, x_n)$ is a Steiner nullvector of order $3$ and $s = \sum_{i=1}^n x_i$, then $s^3$ lies within the Steiner ideal $J$.  
\end{theorem}

\begin{proof}
We can write $p=gs$, where $p = p^{(3)}_T$, $s = \sum_{i} x_i$, and $g = 3 \sum_{i < j} d_T(i,j) x_i x_j$.  Thus, writing $D_r$ for differentiation with respect to $x_r$, we obtain
$$
D_r p = g + s D_r g
$$ 
Then
\begin{align*}
\sum_r x_r D_r p &= \sum_r x_r (g + s D_r g) \\
&= g \sum_r x_r + s \sum_r x_r D_r g \\
&= s(g + \sum_r x_r D_r g).
\end{align*}
Now, 
\begin{align*}
    \sum_r x_r D_r g &= 3 \sum_r x_r D_r \left (\sum_{i < j} d_T(i,j) x_i x_j \right )\\
    &= 3 \sum_r x_r \sum_{j} d_T(r,j) x_j \\    
    &= 6 \sum_{r < j} d_T(r,j) x_r x_j  = 2g.    
\end{align*}
Putting these together gives that $\sum_r x_r D_r p = s(g + 2g) = 3sg$.  So $sg$ is in the Steiner ideal $J = \langle \{D_r p\}_{r} \rangle$.  Since $D_r p = g+  sD_r g \in J$, we also have $s(g+s D_r g) = sg+s^2 D_r g \in J$, and so also $sg+s^2 D_r g-sg = s^2 D_r g \in J$.  

Now, $D_r g = 3 \sum_j d_T(j,r) x_j$.  In other words, $\nabla g = M x$, where $M$ denotes the (symmetric) distance matrix of the tree and $x$ is the vector of all variables.  By the Graham-Pollak Theorem, $M$ is invertible for trees, so $y M = \vec{1}$ has a solution (where $\vec{1}$ is the all-ones row vector).  Let the solution be $y=(c_1,\ldots,c_n)$.  Then
$$ 
y \nabla g = y M x = \vec{1} x = s
$$
i.e., $\sum_r c_r D_r g = s$.  Thus, $\sum_r c_r s^2 D_r g = s^3 \in J$.
\end{proof}

In fact, tracing back through the computation gives $s^3 = \sum_r f_r D_r p$ where
$$
f_r = c_r s - \frac{x_r}{3} \sum_j c_j.
$$
It is not hard to deduce from Proposition \ref{prop:notdivisible} below that $s^2 \not \in J$.

\begin{cor}\label{zerosum} If $\bv$ is a Steiner nullvector, then the sum of the coordinates of $\bv$ is $0$.
\end{cor}
\begin{proof} Since $s^3 \in J$, we have $s \in \sqrt{J}$.  Therefore, if $\bv$ is in the Steiner nullvariety, then $s(\bv) = 0$, i.e., the coordinates of $\bv$ sum to $0$.
\end{proof}

\begin{theorem}[\cite{GraLov78} Lemma 1] \label{thm:GLinversedistance} Let $T$ be a tree with vertex set $[n]$, let $d_j$ be the degree of vertex $j$, and let $a_{ij}$ be the indicator function that $ij \in E(T)$.  If $D$ is the distance matrix of $T$ and the $ij$-entry of $D^{-1}$ is $d^\ast_{ij}$, then
$$
d^\ast_{ij} = \frac{(2-d_i)(2-d_j)}{2(n-1)} + \left \{ \begin{array}{ll}
-d_i/2 & \text{ if } i=j \\
a_{ij}/2 & \text{ if } i \neq j 
\end{array}
\right .
$$
\end{theorem}

\begin{prop} $c_r = (2-d_r)/(n-1)$ and $\sum_r c_r = 2/(n-1)$.
\end{prop}
\begin{proof}
    Let $M$ be the distance matrix of $T$.  Since $yM = \vec{1}$ and $M$ is invertible,
    $$
    y = \vec{1} M^{-1}.
    $$
    Therefore, applying Theorem \ref{thm:GLinversedistance}, we can write
    \begin{align*}
    c_r &= \sum_j \left ( \frac{(2-d_r)(2-d_j)}{2(n-1)} + \left \{ \begin{array}{ll}
-d_r/2 & \text{ if } r=j \\
a_{rj}/2 & \text{ if } r \neq j 
\end{array}  \right . \right ) \\
&= \frac{2-d_r}{2(n-1)} \sum_j (2-d_j) - \frac{d_r}{2} + \frac{d_r}{2} \\
&= \frac{2-d_r}{2(n-1)} (2n - 2(n-1)) = \frac{2-d_r}{n-1}.
    \end{align*}
Thus,
$$
\sum_r c_r = \sum_r \frac{2-d_r}{n-1} = \frac{1}{n-1} (2n - 2(n-1)) = \frac{2}{n-1}.
$$
\end{proof}

 \begin{cor} $s^3 = \sum_r f_r D_r p$ where
$$
f_r = \frac{1}{n-1} \left ( (2-d_r) s - \frac{2}{3}x_r \right )
$$
\end{cor}

So, $s$ is in the radical $\sqrt{J}$ of $J$, and we can write $s^3$ (but not $s^2$) in terms of the generators of $J$.  In particular, the codimension of the Steiner nullvariety is at least one. The next few results show that the codimension is in fact, $2$.

\begin{prop} \label{prop:notdivisible}
    The polynomials $D_r p$ are not divisible by $s$.
\end{prop}
\begin{proof}
    Suppose $s | D_r p$.  Then, since $p = gs$, we have $D_r p = g + s D_r g$, so $s | g$ as well.  But, $g$ is quadratic, so there exist $a_1,\ldots,a_n \in \mathbb{C}$ so that $g = s \sum_{r} a_r x_r$, i.e.,
    $$
    g = \sum_{i,j} a_i x_i x_j.
    $$
    The $x_i^2$ term on the right-hand side is $a_i x_i^2$, but the corresponding coefficient on the left-hand side is $0$, so $a_i = 0$ for each $i$.  Then $f=0$, so $g=0$, a contradiction.
\end{proof}

\begin{theorem}
The codimension of an order-$3$ Steiner nullvariety of a tree is $2$.
\end{theorem}
\begin{proof}
    If $J$ is the Steiner ideal, then, by the previous result, $\langle s \rangle \subsetneq \langle s,g \rangle \subseteq \sqrt{J}$.  On the other hand, $D_r p = g + s D_r g \in \langle g,s \rangle$, so $\sqrt{J} = \langle s,g \rangle$.
\end{proof}

In fact, we can go even further: for every assignment of values to $n-2$ vertices, there is an assignment to the last two vertices that yields a Steiner nullvector:

\begin{cor} For any tree $T$ on $n$ vertices and $n-2$ values $a_3,\ldots,a_{n} \in \mathbb{C}$, there exist $a_1,a_2$ so that $(a_1,\ldots,a_n)$ is a Steiner nullvector. 
\end{cor}
\begin{proof}
    We need only show that the Steiner nullvariety is not contained in any hyperplane of the form $x_r = c$, i.e., no polynomial of the form $x_r - c$ is an element of $\sqrt{J}$.  However, $\sqrt{J} = \langle s,g \rangle$ is a homogeneous ideal of degree $2$, so it does not contain any linear polynomials.
\end{proof}

\begin{prop}
For any tree $T$ on $n$ vertices and $n-2$ values $a_3,\ldots,a_{n} \in \mathbb{C}$, there exist $a_1,a_2$ so that $(a_1,\ldots,a_n)$ is a Steiner nullvector: $a_1$ is any solution to
$$
    A a_1^2 + B a_1 + C = 0,
$$ 
where $A = d_T(1,2)$, $B = \sum_{j \geq 3} (d_T(1,2)-d_T(1,j)+d_T(2,j))a_j$, and $C = \sum_{j,k \geq 3} (d_T(2,j) - 2d_T(j,k)) a_j a_k$; and $a_2 = -a_1 -\sum_{j=3}^na_j.$
\end{prop}

\begin{proof}
    Assume $v = (a_1, \cdots, a_n)$ is a nullvector where $a_3, \cdots, a_n$ are arbitrary complex number. Since $v$ is a nullvector, Corollary \ref{zerosum} states that $\sum_{j=1}^n a_j = 0$. Therefore, $a_2 = -a_1 - \sum_{j=3}^na_j$. 
    
    Also, since $v$ is a nullvector, by definition all partial derivatives to the Steiner 3-form must vanish. Notice by Theorem \ref{divisible}, $D_rp = D_r(sg) = sD_rg + g$ where $g = 3\sum_{j<k}d_T(j,k)a_ja_k$. Since $s = \sum_{j=1}^na_j = 0$, this means that we only need to show that $g = 3\sum_{j<k}d_T(j,k)a_ja_k = 0$. Rewriting $g$ to pull out any terms involving $a_1$ or $a_2$, we see that 
    $$
    3d_T(1,2)a_1a_2+3a_1\sum_{j=3}^nd_T(1,j)a_j + 3a_2\sum_{j=3}^nd_T(2,j)a_j + 3\sum_{\substack{j<k\\j \geq 3}}^nd_T(j,k)a_ja_k = 0.
    $$ 
    Plugging in $a_2 = -a_1 - \sum_{j=3}^na_j$ and simplifying yields
    \begin{align*}
    a_1^2 d_T(1,2) & + a_1 \left [ \sum_{j \geq 3} (d_T(1,2)-d_T(1,j)+d_T(2,j))a_j \right ] \\
    & \qquad + \sum_{j,k \geq 3} (d_T(2,j)-2 d_T(j,k)) a_j a_k = 0,        
    \end{align*}
    which has a solution for every choice of $a_3,\ldots,a_n$.
\end{proof}

\bibliographystyle{plain}
\bibliography{ref}

\end{document}